\documentclass[reqno,makeidx,12pt]{amsart}
        
\usepackage{ifpdf}
\ifpdf
        \pdfoutput=1 
        \pdftrue 
        \fi

\usepackage{graphicx}
\usepackage{latexsym}
\usepackage{amstext}
\usepackage {amsmath}
\usepackage {amsfonts}
\usepackage {amssymb}
\usepackage {amsthm}
\usepackage {bbm}
\usepackage{enumerate}
\usepackage{array}
\usepackage{comment}
\usepackage{xspace}
\usepackage[bookmarks=true]{hyperref}
\usepackage{enumerate}
\ifx\cs\documentclass \footheight 12pt \fi \footskip 30pt 
\DeclareMathOperator{\spt}{supp}

\DeclareMathOperator{\dive}{div}
\def\loc{{\mathrm{loc}}}

\newcommand{\eps}{\varepsilon}

\newcommand{\R}{\ensuremath{\mathbb{R}}}
\newcommand{\Rn}{\ensuremath{{\mathbb{R}^n}}}

\newcommand{\N}{\ensuremath{\mathbb{N}}}
\newcommand{\LL}{\ensuremath{\mathcal{L}}}

\newcommand{\Ha}{\ensuremath{\mathcal{H}}}

\def\R{\mathbb R}

\def\W{\mathcal{W}}
\def\A{\mathcal{S}}
\def\F{\mathcal{F}}

\def\Per{\mathcal{P}}

\theoremstyle{plain}
\numberwithin{equation}{section}
\newtheorem{lemma}{Lemma}[section]
\newtheorem{theorem}[lemma]{Theorem}
\newtheorem{proposition}[lemma]{Proposition}
\newtheorem{definition}[lemma]{Definition}

\theoremstyle{definition}
\newtheorem{remark}[lemma]{Remark}
\newtheorem{notation}[lemma]{Notation}
\begin{document}
\title[Convergence of perturbed Allen--Cahn equations]{Convergence of
  perturbed Allen--Cahn equations to forced mean curvature flow}
\author{Luca Mugnai}
\address{Luca Mugnai, Max Planck Institute for Mathematics in the
  Sciences, Inselstr. 22, D-04103 Leipzig}

\author{Matthias R{\"o}ger}
\address{Matthias R\"{o}ger, Max Planck Institute for Mathematics in the
  Sciences, Inselstr. 22, D-04103 Leipzig}

\email{mugnai@mis.mpg.de, roeger@mis.mpg.de}

\subjclass[2000]{Primary  53C44; Secondary  35K55,  49Q15}

\keywords{Allen-Cahn equation, sharp interface limits, motion by mean curvature}

\date{\today}

\begin{abstract}
We study perturbations of the Allen--Cahn equation and prove the
convergence to forced mean curvature flow in the sharp
interface limit. We allow for perturbations that are
square-integrable with respect to the diffuse surface area measure. We
give a suitable generalized formulation for forced 
mean curvature flow and apply previous results for the Allen--Cahn action
functional. Finally we discuss some applications.
\end{abstract}

\maketitle
\section{Introduction}
\label{sec:intro}
In this paper we study perturbed Allen--Cahn equations of the form 
\begin{align}
  \eps \partial_t u_\eps \,&=\, \eps\Delta u_\eps -\frac{1}{\eps}W'(u_\eps) +
  g_\eps\quad\text{ in }\Omega_T, 
  \label{eq:p-AC}
  \\
  u_\eps(0,\cdot)\,&=\, u_\eps^0\quad\text{ in }\Omega, 
  \label{eq:init}
  \\
  \nabla u_\eps\cdot\nu_\Omega\,&=\, 0\quad\text{ on }(0,T)\times
  \partial\Omega,
  \label{eq:bdry}
\end{align} 
where the spatial domain $\Omega$ is given by an open 
bounded 
set in
$\R^n$ with Lipschitz boundary, $(0,T)$ is a fixed time intervall,
$\Omega_T:=(0,T)\times\Omega$, and
$W$ is the standard quartic double-well potential
\begin{gather*}
  W(r)\,=\,\frac{1}{4}(1-r^2)^2.
\end{gather*}
We are interested in
the asymptotics of {\eqref{eq:p-AC}} in the \emph{sharp interface limit}
$\eps\to 0$ for forcing terms $g_\eps$ that satisfy
\begin{gather}
  \sup_{\eps>0}\int_0^T\int_\Omega\frac{1}{\eps}g_\eps(t,x)^2\,dx
  dt\,=:\,\Lambda\,<\,\infty.
   \label{eq:ass-g}
\end{gather}

Perturbations of this type arise for example in models for
diffusion-induced grain boundary motions \cite{CaFP97}, in models for
phase transitions \cite{BeMi98}, \cite{Wa93},
\cite{Sone95}, and in image processing \cite{BeCM04}.

If $g_\eps=0$ then \eqref{eq:p-AC} reduces to the standard Allen--Cahn
equation. It is well known that in this case the sharp
interface limit is given by the evolution of phase boundaries by mean
curvature flow \cite{BaSS93,EvSS92,Ilma93}. Our goal is to prove that
solutions of the 
perturbed equation \eqref{eq:p-AC} converge to motion by \emph{forced mean
  curvature flow}, 
\begin{gather}
  v\,=\, H + g. \label{eq:fmc}
\end{gather}
Here $v$ describes the velocity vector of an evolution of phase
boundaries $(\Gamma_t)_{t\in
(0,T)}$, $H(t,\cdot)$ denotes the mean curvature vector of $\Gamma_t$, and $g$
is an appropriate limit of $g_\eps \nabla u_\eps$.

Since the limit evolution in general allows for the formation of singularities
in finite time it is necessary to consider suitable generalized
formulations of \eqref{eq:fmc}. In the analysis of mean curvature flow
different techniques have been 
sucessfully applied, in particular viscosity solutions
\cite{BaSS93,ChGG91,CrIL92,EvSS92,EvSp91}, De 
Giorgi's barriers method \cite{BeNo97,BePa95,ChNo08,DiLN01}, and
geometric measure theory formulations. We follow
here the latter approach and use in particular many ideas
from the work of Brakke 
\cite{Brak78} and Ilmanen \cite{Ilma93} on mean curvature flow and the
convergence of the Allen--Cahn equation, respectively. To avoid problems with
cancellations of phase boundaries we consider not only the evolution of
the phases but also the evolution of certain \emph{energy measures}. In the case
of a smooth limit evolution and `nicely behaving' 
approximations these measures coincide with the surface area
measures associated with the phase boundaries, but in 
general they may be supported on additional \emph{hidden boundaries} or
may carry a higher mulitplicity. Generalizing hypersurfaces in this way
is in the spirit of the theory of (integral) varifolds, which allows to give a meaning to
geometric quantities such as mean curvature and second
fundamental form, and which provides good compactness properties. In the
context of phase transition problems this technique has been successfully applied
to a couple of different problems
\cite{Chen96,Sone95,HuTo00,Tone05,RSc06,MuR08,RTo08}.  

Our main result is the convergence of solutions to 
\eqref{eq:p-AC} to an \emph{$L^2$-flow} of energy measures that move by
forced mean curvature flow.  
The concept of $L^2$-flows
was develloped in \cite{MuR08} and describes an
evolution of integral varifolds with square integrable weak mean curvature
and square integrable generalized velocity. We verify the evolution law
\eqref{eq:fmc} in a pointwise formulation almost everywhere
with respect to the energy measures. For a precise formulation of our main
result see Section \ref{sec:main}.

One benefit of our approach is that we do not use a comparison
principle neither for the perturbed Allen--Cahn equation nor for the
forced mean curvature flow. This makes our technique quite flexible 
compared to viscosity solution approaches or to the use of maximum
principles to prove the non-positivity of the discrepancy measures, as
pursued 
in \cite{Ilma93}. Compared to previous results on forced mean curvature
flow \cite{BaSS93,BaSo98,ChNo08} and on the 
convergence of perturbed Allen--Cahn equations our results are more general
in the regularity that is required for the forcing term. We do only need
that the forcing term is (uniformly) $L^2$-integrable with respect to the
(diffuse) surface energy measures.  
On the other hand our proof is limited to space dimensions $n=2,3$ and our formulation
of the limit equation is weaker.

This paper borrows many ideas from our analysis of the Allen--Cahn
\emph{action functional} \cite{MuR08}, which is defined for any 
smooth function $u:\Omega_T\to\R$ by
\begin{gather}
 \A_\eps(u)\,:=\, \int_0^T\int_\Omega \Big(\sqrt{\eps}\partial_t u
  +\frac{1}{\sqrt{\eps}}\big(-\eps\Delta u 
  +\frac{1}{\eps}W^\prime(u)\big)\Big)^2\,dx\,dt. \label{def:action}
\end{gather}
The functional $\A_\eps$ is connected to the small noise limit of the probability
of rare events in the stochastically perturbed Allen--Cahn equation
\cite{KORV07}.  
The assumption \eqref{eq:ass-g} on $g_\eps$ yields a uniform bound on
the action for solutions $(u_\eps)_{\eps>0}$ of \eqref{eq:p-AC}. By
\cite{MuR08} this implies the convergence of diffuse surface area
measures associated to $u_\eps$ to an $L^2$-flow in the limit $\eps\to
0$. In this paper we discuss the convergence of the evolution laws
and present some applications. In particular, we 
prove the convergence of diffuse approximations to the Mullins--Sekerka
problem with kinetic undercooling in dimensions $n=2,3$, which improves
earlier results by Soner \cite{Sone95}. 

An important ingredient to derive the compactness of action-bounded
sequences and solutions of \eqref{eq:p-AC} is stated in a
(modified) conjecture of De Giorgi \cite{DeG91}: Considering 
\begin{gather}
E_\eps(u):=\int_\Omega\left(\frac{\eps}{2}\vert\nabla u\vert^2+\frac{W(u)}{\eps}\right)\, dx,
\\
 \W_\eps(u)\,:=\,\int_\Omega \frac{1}{\eps}\Big(-
  \eps\Delta u +\frac{1}{\eps}W'(u)\Big)^2\,dx \label{eq:def-W-intro}
\end{gather}
the sum $E_\eps+\W_\eps$ Gamma-converges, up to a constant factor $c_0$, 
to the sum of the Perimeter functional $\Per$ and the \emph{Willmore functional} $\W$,
\begin{gather}
  E_\eps + \W_\eps\,\to\, c_0\Per + c_0\W,\qquad
  \W(u)\,=\, \int_{\Gamma} H^2\,d\Ha^{n-1}, \label{eq:def-willmore}
\end{gather}
where $\Gamma$ denotes the phase boundary $\partial^*\{u=1\}\cap\Omega$
and where
\begin{gather}\label{eq:def-c0}
  c_0 \,:=\, \, \int_{-1}^1 \sqrt{2W(s)}\,ds.
\end{gather}
This statement was proved in space dimensions $n=2,3$ by R\"oger and
Sch\"atzle \cite{RSc06} and provides a diffuse version of Allard's
compactness theorem for integral 
varifolds (in the special case of a uniform $L^2$ bound on the mean
curvature and $n=2,3$). In particular we avoid the use of a diffuse
version of Huisken's monotonicity formula \cite{Huis90} to derive the rectifiability of
the limiting energy measures, as it was done in \cite{Ilma93}.
\section{$L^2$-flows and diffuse surface area measures}
\label{sec:l2flows}
In this section we state our weak formulation for
evolutions of mean curvature flow type. For basic notions from
geometric measure theory we refer to \cite{Alla72,Simo83}.
\begin{notation}
A (general) varifold on $\Omega$ is a Radon measure on the Grassmannian
$G^{n-1}(\Omega)$, i.e. the euclidean product of $\Omega$ with the space of unoriented
$(n-1)$ planes in $\Rn$.
A Radon measure $\mu$ on $\Omega$ is $(n-1)$-integer rectifiable if in
$\mu$-almost all points $x\in\Omega$ the $(n-1)$-dimensional (measure
theoretical) tangent plane $T_x\mu$ exists and if $\mu$-almost
everywhere the $(n-1)$-dimensional density $\theta^{n-1}(\mu,\cdot)$ is
integer-valued. A varifold 
$V$ on $\Omega$ is $(n-1)$-integer rectifiable if there exists an
$(n-1)$-integer rectifiable Radon measure $\mu$ on $\Omega$ such that
\begin{gather*}
  \int_{G^{n-1}(\Omega)} \zeta(x,S)\,dV(x,S)\,=\, \int_\Omega
  \zeta(x,T_x\mu)\,d\mu(x) 
\end{gather*}
for all $\zeta\in C^0_c(G^{n-1}(\Omega))$. This gives a
one-to-one correspondence between $(n-1)$-integer 
rectifiable varifolds and $(n-1)$-integer
rectifiable Radon measure on $\Omega$. In this paper we will identify
the corresponding objects and use the term \emph{integral varifold}.

The first variation $\delta\mu$ of an integral varifold $\mu$ in direction of a
vector field $\eta\in C^1_c(\Omega,\Rn)$ is defined by
\begin{gather*}
  \delta\mu(\eta)\,:=\, \int_\Omega \dive_{T_x\mu}\eta(x)\,d\mu(x),
\end{gather*}
where $\dive_{T_x\mu}$ denotes the divergence restricted to the
$(n-1)$-plane $T_x\mu$. We say that $\mu$ has a \emph{weak mean curvature vector}
$H_\mu\in L^1_{\loc}(\mu,\,\R^n)$ if for all $\eta\in C^1_c(\Omega,\,\R^n)$ the first
variation is given by
\begin{gather*}
  \delta\mu(\eta)\,=\, -\int_\Omega H(x)\cdot\eta(x)\,d\mu(x).
\end{gather*}

For a family of measures $(\mu^t)_{t\in (0,T)}$ we denote by $\LL^1\otimes
\mu^t$ the product measure defined by
\begin{gather*}
  \big(\LL^1\otimes \mu^t\big)(\eta)\,:=\, \int_0^T \mu^t(\eta(t,\cdot))\,dt 
\end{gather*}
for any $\eta\in C^0_c(\Omega_T)$.
\end{notation}
We next recall the definition and basic properties of \emph{$L^2$-flows}
\cite{MuR08}, which describe evolutions of integral
varifolds with square integrable mean curvature vector and square
integrable \emph{generalized velocity}.
\begin{definition}
\label{def:gen-velo}
Let $(\mu^t)_{t\in (0,T)}$ be any family of integral varifolds
such that $\mu:=\LL^1\otimes\mu^t$ defines a Radon measure on $\Omega_T$ and
such that $\mu^t$ has a weak mean curvature $H(t,\cdot)\in L^2(\mu^t,\,\R^n)$
for almost all $t\in (0,T)$.

If there exists a positive
constant $C$ and 
a vector field $v\in L^2(\mu,\Rn)$ such that
\begin{gather}
  v(t,x)\perp T_x\mu^t\quad\text{ for }\mu\text{-almost all
  }(t,x)\in\Omega_T, 
  \label{eq:def_vel_perp}
  \\
  \label{eq:def_vel_L2_flow}
  \Big\vert \int_0^T\int_\Omega \big(\partial_t\eta+
  \nabla\eta\cdot v\big)\,d\mu^tdt \Big\vert \,\leq\,
  C\|\eta\|_{C^0(\Omega_T)}
\end{gather}
for all $\eta\in C^1_c((0,T)\times{\Omega})$, then we call  the
evolution 
$(\mu^t)_{t\in (0,T)}$ an \emph{$L^2$-flow}. A function $v\in
L^2(\mu,\Rn)$ satisfying \eqref{eq:def_vel_perp},
\eqref{eq:def_vel_L2_flow} is called a 
\emph{generalized velocity vector}.
\end{definition}
\begin{remark}
This definition is based on the observation that for a smooth evolution
$(M_t)_{t\in (0,T)}$ with square-integrable mean curvature $H(t,\cdot)$ and
square-integrable normal
velocity vector $V(t,\cdot)$
\begin{align*}
  &\frac{d}{dt}\int_{M_t}\zeta(t,x)\,d\Ha^{n-1}(x) -
  \int_{M_t}\partial_t\zeta(t,x)\,d\Ha^{n-1}(x)  -
  \int_{M_t}\nabla\zeta(t,x)\cdot V(t,x)\,d\Ha^{n-1}(x) \\
  =\,&  \int_{M_t}H(t,x)\cdot V(t,x)\zeta(t,x)\,d\Ha^{n-1}(x).
\end{align*}
Integrating this equality in time and using H\"olders inequality on the
right-hand side implies
\eqref{eq:def_vel_L2_flow}.
\end{remark}

Any
generalized velocity is (in a set of good points) uniquely determined by
the evolution $(\mu^t)_{t\in (0,T)}$.  In particular, in the case that
$(\mu^t)_{t\in (0,T)}$ describes a smooth evolution of smooth
hypersurfaces the generalized velocity coincides with the classical
velocity of hypersurfaces.
\begin{proposition}\cite[Proposition 3.3]{MuR08}
\label{prop:uni_gen_vel_perp} 
Let $(\mu^t)_{t\in(0,T)}$ be an $L^2$-flow and set
$\mu:=\LL^1\otimes\mu^t$. Let $v\in L^2(\mu,\,\R^n)$ be a generalized velocity
field in the sense of Definition \ref{def:gen-velo}. 
Then 
\begin{equation}\label{eq:gen_v_perp}
  \begin{pmatrix} 1 \\ v(t_0,x_0)\end{pmatrix} \in T_{(t_0,x_0)}\mu
\end{equation}
holds in $\mu$-almost all points $(t_0,x_0)\in\Omega_T$ where the
tangential plane of $\mu$ exists. The evolution
$(\mu^t)_{t\in(0,T)}$ uniquely determines $v$ in all points
$(t_0,x_0)\in \Omega_T$ where both tangential planes
$T_{(t_0,x_0)}\mu$ and $T_{x_0}\mu^{t_0}$ exist.
\end{proposition}
We next define on the level of phase field approximation \emph{diffuse surface area
  measures}.
\begin{definition}\label{def:mu}
For $\eps>0$, $t\in (0,T)$ define a  Radon measures $\mu_\eps^t$ on
${\Omega}$, 
\begin{align}
  \mu_\eps^t \,&:=\, \Big(\frac{\eps}{2}|\nabla u_\eps|^2(t,\cdot)
  +\frac{1}{\eps}W(u_\eps(t,\cdot))\Big)\LL^n, \label{eq:def-mu-eps,t}
\end{align}
and for $\eps>0$ Radon measures $\mu_\eps$ on ${\Omega_T}$,
\begin{align}
  \mu_\eps \,&:=\, \Big(\frac{\eps}{2}|\nabla u_\eps|^2
  +\frac{1}{\eps}W(u_\eps)\Big)\LL^{n+1}. \label{eq:def-mu-eps}
\end{align}
\end{definition}
We will show, that the surface area measures $\mu_\eps^t$ converge in the
limit $\eps\to 0$ to an $L^2$-flow.

To the diffuse surface energy measures $\mu_\eps^t$ we associate
a normal direction $\nu_\eps(t,\cdot)$, varifolds $V_\eps^t$, and Radon measures
$\tilde{\mu}_\eps$, respectively, defined by 
\begin{align}
  \nu_\eps\,&:= 
  \begin{cases}
    \frac{\nabla u_\eps}{|\nabla u_\eps|} &\text{ if } \vert\nabla
    u_\eps|>0,\\
    \vec{e}_1 &\text{ else,}
  \end{cases}
  \label{eq:def-normal}\\
  V_\eps^t(\eta)\,&:=\, \int_\Omega
  \eta(x,\nu_\eps(t,x)^\perp)\,d\mu_\eps^t(x) \qquad\text{ for }\eta\in
  C^0_c(\Omega\times G_{n,n-1}), \label{eq:def-Veps}\\
  \tilde{\mu}_\eps\,&:=\, \eps|\nabla u_\eps|^2\,\LL^{n+1}. \label{eq:def-tilde-mu}
\end{align}
\section{Main results}
\label{sec:main}
\begin{theorem}\label{the:main}
Let $n=2,3$ and let $(u_\eps)_{\eps>0}$, $(g_\eps)_{\eps>0}$, and $(u_\eps^0)_{\eps>0}$
 be given such that
\eqref{eq:p-AC}-\eqref{eq:ass-g} holds, and such that 
\begin{gather}
  \sup_{\eps>0}\mu_\eps^0(\Omega)\,=:\, \Lambda_0<+\infty. 
  \label{eq:ass-ini}
\end{gather}
Then there exists a
subsequence $\eps\to 0$, 
a phase indicator function $u\in BV(\Omega_T)\cap
L^\infty(0,T;BV(\Omega;\{-1,1\})$, an $L^2$-flow $(c_0^{-1}\mu^t)_{t\in
(0,T)}$, and a vector-field $g\in L^2(\mu;\Rn)$ such that the following
properties hold: 
\begin{enumerate}
\setlength{\leftmargin}{-10ex}
\item
Convergence of phase fields
 \begin{alignat}{2}
  u_\eps\,&\to\, u&&\text{ in }L^1(\Omega_T),\label{eq:conv-MM}\\
  u_\eps(t,\cdot)\,&\to\, u(t,\cdot)\qquad&&\text{ in }L^1(\Omega)\text{
  for all }t\in [0,T]. \label{eq:conv-MM-t}
\end{alignat}
Moreover, $u\in C^{0,\frac{1}{2}}([0,T],L^1(\Omega))$.
\item
Convergence of diffuse surface area measures
\begin{alignat}{2}
  \mu_\eps\,&\to\, \mu \quad&&\text{ as Radon-measures on
  }{\Omega_T}, \label{eq:conv-mu}\\
  \mu_\eps^t\,&\to\, \mu^t&&\text{ for all }t\in [0,T], \text{ as Radon
  measures on }{\Omega}, 
  \label{eq:conv-mu-t}\\
  \mu^t\,&\geq\, \frac{c_0}{2}|\nabla u(t,\cdot) |. \label{eq:spt-mu}
\end{alignat}
\item
Convergence of force fields
\begin{gather}
  \lim_{\eps\to 0}\int_{\Omega_T} -\eta\cdot\nabla u_\eps
  g_\eps\,dx\,dt\,=\, \int_{\Omega_T} \eta\cdot g\,d\mu
  \label{eq:conv-g}
\end{gather}
holds for all  $\eta\in C^0_c(\Omega_T)$.
\item
Motion law.
\begin{gather}
  v \,=\, H + g\label{eq:lim-AC}
\end{gather}
holds $\mu$-almost everywhere, where $H(t,\cdot)$ denotes the weak mean
curvature of $\mu^t$ and where $v$ denotes the generalized velocity of
$(\mu^t)_{t\in (0,T)}$ in the sense of Definition \ref{def:gen-velo}.
\end{enumerate}
\end{theorem}
\begin{remark}\label{rem:undercurrent}
Our formulation of forced mean curvature flow can be understood as a
generalized formulation of forced mean curvature flow \emph{for phase
boundaries}. In particular, $H$, $v$, and $g$ restricted to the phase
interface $\partial^*\{u=1\}$ are a property of $u$ and do not depend on $\mu$.
In fact, by \cite[Proposition 4.5]{MuR08} we know  that $V:=
v\cdot\frac{\nabla u}{|\nabla u|}$ belongs to $L^1(\vert\nabla u\vert)$  and
that for every $\eta\in C^1_c(\Omega_T)$ 
\begin{equation*}
\int_0^T\int_\Omega V(t,\cdot)\eta(t,\cdot)\,d|\nabla u(t,\cdot)|\, dt
=-\int_{0}^T\int_\Omega u(t,x)\partial_t\eta(t,x)\,dxdt,
\end{equation*}
holds. 
By \eqref{eq:gen_v_perp} and the rectifiability of $\partial^* \{u=1\}$ we
conclude that $V$ is $\Ha^n$-almost everywhere on $\partial^* \{u=1\}$
uniquely determined 
by $u$. Moreover, by
\cite{Menn08} we have that $H(t,\cdot)$ restricted to $\partial^* \{u(t,\cdot)=1\}$ is
$\Ha^{n-1}$-almost everywhere uniquely determined by $u(t,\cdot)$. Therefore
\eqref{eq:lim-AC} shows that also $g$ restricted to $\partial^* \{u=1\}$ is a
property of $u$.
\end{remark}
\begin{remark}\label{rem:unbounded}
Our main results hold also if $\Omega$ is
unbounded, with the only change that the limiting phase field $u$
does not belong to $L^1(\Omega_T)$ but for all $D\Subset \Omega$ to
$L^1((0,T)\times D)$ and that the convergence in \eqref{eq:conv-MM},
\eqref{eq:conv-MM-t} is for all $D\Subset \Omega$ in $L^1((0,T)\times D)$ and
$L^1(D)$, respectively. To obtain the corresponding compactness properties, in
Proposition \ref{prop:MM} we use in the case that $\Omega$ is bounded
the Modica--Mortola result \cite{MoMo77} on the Gamma convergence of the
diffuse area. If $\Omega$ is unbounded one uses the Gamma convergence
with respect to the $L^1_{loc}$ topology \cite{DaMo81} instead.
\end{remark}
\section{Proof of Theorem \ref{the:main}}
At several instances we will pass to subsequences $\eps\to 0$ without
relabelling.
Many arguments are from
\cite{MuR08}; to make the paper
self-consistent we give in any case at least a sketch of proof.
\subsection{Uniform estimates}
\label{sec:est}
 By our
assumption \eqref{eq:ass-g} we have a uniform bound 
on the action functional for the sequence $(u_\eps)_{\eps>0}$, i.e.
\begin{gather*}
  \A_\eps(u_\eps)\,=\, \int_{\Omega_T} \Big(\sqrt{\eps}\partial_t u
  +\frac{1}{\sqrt{\eps}}\big(-\eps\Delta u 
  +\frac{1}{\eps}W^\prime(u)\big)\Big)^2\,dx\,dt\,=\, \int_{\Omega_T}
  \frac{1}{\eps}g_\eps^2\,dxdt\,\leq\, \Lambda.
\end{gather*}
In particular all results from \cite{MuR08} apply. 
We next prove uniform bounds for the surface area and for diffuse analog
of the velocity and mean curvature. We
denote by $w_\eps$ the \emph{chemical potential}, that is the  $L^2$-gradient of the diffuse
surface energy,
\begin{gather*}
  w_\eps\,:=\, -\eps\Delta u_\eps 
  +\frac{1}{\eps}W^\prime(u_\eps).
\end{gather*}
From \eqref{eq:bdry}, \eqref{eq:ass-g} we deduce that for any $0\leq
t_0\leq T$ 
\begin{align*}
  \Lambda\geq\,&\int_0^{t_0}\int_\Omega \frac{1}{\eps}\Big({\eps}\partial_t u_\eps
  + w_\eps\Big)^2\,dxdt\\
  =\,&
  \int_0^{t_0}\int_{\Omega} \Big(\eps (\partial_t u_\eps)^2 +\frac{1}{\eps}
  w_\eps^2\Big)\,dxdt + 2 E_\eps(u_\eps(t_0,\cdot)) -2 E_\eps(u_\eps(0,\cdot)),
\end{align*}
which, by \eqref{eq:ass-g} and \eqref{eq:ass-ini}, yields that
\begin{gather}
  \int_{\Omega_T} \Big(\eps (\partial_t u_\eps)^2 +\frac{1}{\eps} w_\eps^2
  \Big)\,dxdt \,\leq\, \Lambda+ 2\Lambda_0, 
  \label{ass:bound-squares}\\ 
  \max_{0\leq t\leq T} E_\eps(u_\eps(t,\cdot)) \,\leq\, \frac{1}{2}\Lambda + \Lambda_0.
  \label{ass:bound-mu}
\end{gather}
We next derive an estimate on the change of the diffuse surface area \emph{measures}
in time.
\begin{lemma}
\label{lem:bv-mu}
There exists $C=C(\Lambda,\Lambda_0)$ such that for all
$\psi\in C^1(\overline{\Omega})$ 
\begin{gather}
  \int_0^T |\partial_t \mu_\eps^t(\psi)|\,dt\,\leq\,
  C\|\psi\|_{C^1(\overline{\Omega})}. \label{eq:bv-mu}
\end{gather}
\end{lemma}
\begin{proof}
Using \eqref{eq:bdry} we compute that 
\begin{align}
  2\partial_t\mu_\eps^t(\psi) \,
  =\, &\int_\Omega \frac{1}{\eps}g_\eps(t,x)^2 \psi(x)\, dx - \int_{\Omega}
  \big(\eps(\partial_t u_\eps)^2 
  +\frac{1}{\eps}w_\eps^2\big)(t,x)\psi(x) \,dx\notag\\
  &  - 2\int_{\Omega} \eps\nabla\psi(x)\cdot \partial_t u_\eps(t,x)\nabla
  u_\eps(t,x)\, dx. \label{eq:KRT-1} 
\end{align}
By \eqref{ass:bound-squares}, \eqref{ass:bound-mu} we obtain
\begin{align}
  \Big|2\int_{\Omega_T} \eps\nabla\psi\cdot \partial_t u_\eps\nabla
  u_\eps\,dxdt\Big|\,\leq\,&
  \int_{\Omega_T} |\nabla\psi|\big(\eps (\partial_t u_\eps)^2 +\eps
  |\nabla u_\eps|^2\big)\,dxdt
  \notag\\
  \leq\,&
  C(\Lambda,\Lambda_0,T)\|\nabla\psi\|_{C^0(\overline{\Omega})} \label{eq:KRT-2}
\end{align}
and deduce from \eqref{eq:ass-g}, \eqref{ass:bound-squares},
\eqref{eq:KRT-1}, \eqref{eq:KRT-2} that 
\begin{gather*}
  \int_0^T |\partial_t\mu_\eps^t(\psi)|\,dt\,\leq\,
  C(\Lambda,\Lambda_0,T) \Big(\|\psi\|_{C^0(\overline{\Omega})} +
  \|\nabla\psi\|_{C^0(\overline{\Omega})}\Big),
\end{gather*}
which proves \eqref{eq:bv-mu}.  
\end{proof}
%
\subsection{Convergence of phase fields and diffuse surface area measures}
\label{sec:proof-phases}
\begin{proposition}\label{prop:MM}
There exists a subsequence $\eps\to 0$ and a phase indicator function $u\in
BV(\Omega_T,\{-1,1\})\cap L^\infty(0,T;BV(\Omega))$ 
such that \eqref{eq:conv-MM}, \eqref{eq:conv-MM-t} hold.
Moreover $u\in C^{0,1/2}(0,T;L^1(\Omega))$.
\end{proposition}
\begin{proof}
By \eqref{ass:bound-squares}, \eqref{ass:bound-mu} we can apply the
compactness and lower bound from Modica and
Mortola \cite{MoMo77,Modi87} in the time-space domain and obtain the
existence of $u\in BV(\Omega_T,\{-1,1\})$ such that \eqref{eq:conv-MM}
holds for a subsequence $\eps\to 0$ and such that $u_\eps(t,\cdot)\,\to\,
u(t,\cdot)$ for almost all $t\in (0,T)$. Next we apply again
\cite{MoMo77,Modi87} and we obtain from \eqref{ass:bound-mu} that
$u(t,\cdot)\in BV(\Omega)$ with uniformly bounded $BV$-norm.
Moreover, for the function
\begin{gather}
  G(r)\,:=\, \int_0^r \sqrt{2W(s)}\,ds \label{eq:def-H}
\end{gather}
and almost all $t\in (0,T)$ we have $G(u_\eps(t,\cdot))\,\to\, \frac{c_0}{2}
u(t,\cdot)$ and hence for almost all $t_1,t_2$
\begin{align*}
  \frac{c_0}{2}\int_\Omega |u(t_1,x)- u(t_2,x)|\,dx\,&=\, \lim_{\eps\to 0}
  \int_\Omega |G(u_\eps(t_1,x))- G(u_\eps(t_2,x))|\,dx\\
  &=\, \int_\Omega \Big|\int_{t_1}^{t_2} \partial_\tau
  G(u_\eps(\tau,x))\,d\tau\Big|\,dx\\
  &=\, \int_\Omega \Big|\int_{t_1}^{t_2} \partial_\tau u_\eps(\tau,x)
  \sqrt{2W(u_\eps(\tau,x))} \,d\tau\Big|\,dx\\
  &\leq\, \Big(\int_{\Omega_T} \eps (\partial_\tau
  u_\eps(\tau,x))^2\,dx d\tau\Big)^{1/2}
  \Big|\int_{t_1}^{t_2}\int_\Omega \frac{2}{\eps}W(u_\eps(t,x))\,dx dt
  \Big|^{1/2} \\
  &\leq \sqrt{|t_1-t_2|} C(\Lambda,\Lambda_0,T),
\end{align*}
where we have used \eqref{ass:bound-squares}, \eqref{ass:bound-mu}.
This shows that $u$ can be extended to $u\in
C^{0,1/2}([0,T],L^1(\Omega))$. Then we also obtain
$u_\eps(t,\cdot)\,\to\, u(t,\cdot)$ for all $t\in (0,T)$.
\end{proof}
We next prove the convergence of the surface area measures $\mu_\eps$ and
$\mu_\eps^t$ for almost all times.
\begin{proposition}\label{prop:KRT}
There exists a subsequence $\eps\to 0$
and Radon measures $\mu^t, t\in [0,T]$, 
such that \eqref{eq:conv-mu-t}, \eqref{eq:spt-mu} hold,
such that
\begin{gather}
  \mu_\eps\,\to\, \mu\quad\text{ as Radon measures on }\Omega_T,\qquad
  \mu\,=\, \LL^1\otimes \mu^t, \label{eq:disint-mu}
\end{gather}
and such that for all $\psi\in
C^1(\overline{\Omega})$ the function 
\begin{gather}
  t\mapsto \mu^t(\psi)\quad\text{ is of bounded variation in }(0,T).
  \label{eq:mu-bv} 
\end{gather}
\end{proposition}
\begin{proof}\cite[Proposition 4.2]{MuR08}
By \eqref{ass:bound-mu} we see that $\mu_\eps$ is uniformly bounded.
Therefore we can select a subsequence $\eps\to 0$ such that
$\mu_\eps\to\mu$ for a Radon-measures $\mu$ on ${\Omega_T}$. 
Choose next a countable family $(\psi_i)_{i\in\N}\subset
C^1(\overline{\Omega})$ which is dense in $C^0(\overline{\Omega})$. By Lemma
\ref{lem:bv-mu} and a diagonal-sequence argument there exists a
subsequence $\eps\to 0$ such that
for all $i\in\N$ and almost all $t\in (0,T)$
\begin{gather}
 m_i(t)\,:=\, \lim_{\eps\to 0}\mu_\eps^t(\psi_i)\text{ exists
   and}\label{eq:conv-mi}\\
 t\mapsto m_i(t)\text{ is of bounded variation on
   }(0,T). \label{eq:conv-mu-i}
\end{gather}
Next it is possible to extend \eqref{eq:conv-mi} to the co-countable subset of
$[0,T]$ where none of the functions $m_i$ has a jump. In a second step
we can pass to a subsequence such that \eqref{eq:conv-mi} holds for all
$t\in [0,T]$.
Taking now an arbitrary $t\in [0,T]$
by \eqref{ass:bound-mu} there exists a subsequence $\eps\to 0$ such that
\begin{gather}
  \mu^t \,:=\, \lim_{\eps\to 0} \mu_\eps^t\quad \text{ exists.}
  \label{eq:conv-mu-eps-t} 
\end{gather}
Hence $\mu^t(\psi_i)=m_i(t)$ and since $(\psi_i)_{i\in\N}$ is
dense in $C^0(\overline{\Omega})$ we can identify all limit points of
$(\mu_\eps^t)_{\eps>0}$ and obtain \eqref{eq:conv-mu-eps-t} for the
whole sequence above and for all
$t\in [0,T]$, which proves \eqref{eq:conv-mu-t}. Moreover, by
\cite{Modi87} we have 
\begin{align*}
  \lim_{\eps\to 0} \mu_\eps^t\,\geq\, \limsup_{\eps\to 0} |\nabla
  G(u_\eps(t,\cdot))|\,\geq\, |\nabla G(u(t,\cdot))|\,=\,
  \frac{c_0}{2}|\nabla u(t,\cdot)|.
\end{align*}

By the Dominated Convergence Theorem we conclude that for any $\zeta\in
C^0(\overline{\Omega_T})$ 
\begin{gather*}
  \int_{\Omega_T} \zeta\,d\mu\,=\, \lim_{\eps\to 0}
  \int_{\Omega_T} \zeta\,d\mu_\eps \,=\,  \lim_{\eps\to 0}\int_0^T
  \int_\Omega \zeta(t,x)\,d\mu_\eps^t(x)\,dt\,=\, \int_0^T \int_\Omega
  \zeta(t,x)\,d\mu^t(x)\,dt,
\end{gather*}
which implies that   $\mu$ decomposes as in \eqref{eq:disint-mu}.
\end{proof}
%
\subsection{Integrality of the limit measures}
One key result is that the limits $\mu^t$ of the
diffuse surface area
measures are not just Radon measures but geometric objects in the sense
of being, up to a
constant, integer-rectifiable with a weak mean curvature $H(t)\in
L^2(\mu^t)$. To derive this property we crucially use a compactness
property and lower bound proved in \cite{RSc06}.
\begin{proposition}\label{prop:lsc-H}
For almost all $t\in (0,T)$ 
\begin{align*}
  &\frac{1}{c_0}\mu^t \text{ is an integral }(n-1)\text{-varifold},\\
  &\mu^t\text{ has weak mean curvature }{H}(t,\cdot)\in 
  L^2(\mu^t,\R^n).
\end{align*}
Moreover,
\begin{gather}
  \int_{\Omega_T} |H|^2\,d\mu\,\leq\, \liminf_{\eps\to 0}
  \int_{\Omega_T} \frac{1}{\eps}w_\eps^2\,dx dt\,<\,C(\Lambda,\Lambda_0)
  \label{eq:L2-H} 
\end{gather}
and for the Radon measures $\tilde{\mu}_\eps$ defined
in  \eqref{eq:def-tilde-mu} we obtain that
\begin{align}
  \tilde{\mu}_\eps\,\,&\to\, \mu\quad\text{ as Radon measures on
  }\Omega_T. \label{eq:conv-tilde-mu}
\end{align}
Finally we deduce that for all $\eta\in C^0_c(\Omega,\Rn)$
\begin{gather}
  \lim_{\eps\to 0}\int_{\Omega_T} \eta(t,x)\cdot
  w_\eps(t,x) \nabla u_\eps(t,x)\,dx dt \,=\, \int_{\Omega_T} \eta(t,x)\cdot
  H(t,x)\,d\mu(t,x). \label{eq:conv-H}
\end{gather}
\end{proposition}
\begin{proof}
By \eqref{ass:bound-squares}, \eqref{ass:bound-mu} and Fatou's Lemma 
\begin{gather}\label{eq:bound-fatou}
\liminf_{\eps\to 0} \Big( \mu^t_\eps(\Omega) + \frac{1}{\eps}\int_\Omega
w^2_\eps(x,t)\,dx\Big)<+\infty,
\end{gather}
holds for $\mathcal L^1$-a.e. $t\in (0,T)$. For every $t\in (0,T)$ such
that \eqref{eq:bound-fatou} hold we can apply  
\cite[Theorems 4.1 and 5.1, Proposition 4.9]{RSc06}. Hence there exists a subsequence
$\eps_i\to 0\, (i\to\infty)$ (that may depend on $t$) such that
$\mu_{\eps_i}^t$ converge as Radon measures (by
\eqref{eq:conv-mu-t} the limit is given by $\mu^t$),
and such that
\begin{gather}
  \frac{1}{c_0}\mu^t\, \text{ is an integral varifold with weak mean
  curvature }H(t,\cdot)\in L^2(\mu^t), \notag\\
 \int_\Omega \vert H(t,x)\vert^2\,d\mu^t(x)\leq
 \liminf_{\eps\to 0}\frac{1}{\eps}\int_\Omega  w_\eps(t,x)^2\,dx. \label{eq:liminf-H}
\end{gather}
Moreover, for any subsequence $\eps_i\to 0\, (i\to\infty)$ such that
\eqref{eq:bound-fatou} is satisfied by \cite[Theorem 4.1, Proposition 4.9]{RSc06} we have
\begin{gather}
 V_{\eps_i}^t \,\overset{i\to\infty}{\to}\, \mu^t\quad  \text{ as varifolds on }\Omega,
 \\
 \frac{\eps_i}{2}|\nabla u_{\eps_i}(t,\cdot)|^2
  -\frac{1}{\eps_i}W(u_{\eps_i}(t,\cdot))\,\to\, 0\quad\text{ in
 }L^1(\Omega), \label{eq:conv-xi-t}
\end{gather}
and by \cite[Proposition 4.10]{RSc06} for any $\psi\in C^1_c(\Omega)$
\begin{align}
  \lim_{i\to\infty} \int_\Omega \psi(x)\cdot w_{\eps_i}(t,x)\nabla
  u_{\eps_i}(t,x)\,dx 
  \,=&\, -\lim_{i\to\infty} \delta V_{\eps_i}^t(\psi)\,
   \label{eq:conv-xi-t-2}
  \\
  =&\, -\delta\mu^t(\psi)
  \,=\, \int_\Omega \psi(x)\cdot H(t,x)\,d\mu^t(x).
  \notag
\end{align}
By a refined version of Lebesgue's Dominated Convergence Theorem it
follows \cite[Proposition 6.1]{MuR08} that
\begin{align}
  \frac{\eps}{2}|\nabla u_{\eps}|^2
  -\frac{1}{\eps}W(u_{\eps})\,&\to\, 0\quad\text{ in
 }L^1(\Omega_T), \label{eq:conv-xi}\\
 \int_{\Omega_T} \eta\cdot w_{\eps}\nabla  u_{\eps}\,dx dt 
  \,&\to\, \int_{\Omega_T} \psi\cdot H\,d\mu. \label{eq:conv-H-rep} 
\end{align}
In particular we deduce \eqref{eq:conv-tilde-mu} from \eqref{eq:conv-xi}
and \eqref{eq:conv-mu} since 
\begin{gather*}
  \lim_{\eps\to 0} \tilde{\mu}_\eps\,=\, \lim_{\eps\to 0} \Big(\mu_\eps +
  \Big(\frac{\eps}{2}|\nabla u_{\eps}|^2 
  -\frac{1}{\eps}W(u_{\eps})\Big)\LL^{n+1}\Big)\,=\, \mu.
\end{gather*}
Finally \eqref{eq:L2-H} follows from \eqref{eq:liminf-H}, Fatous Lemma,
and \eqref{ass:bound-squares}. 
\end{proof}
%
\subsection{Existence of a generalized velocity}
\begin{lemma}
\label{lem:velo}
There exists a function $v\in L^2(\mu,\Rn)$ such that
\begin{gather}
  \lim_{\eps\to 0}\int_{\Omega_T} -\eta(t,x)\cdot
  \eps\partial_t u_\eps(t,x)\nabla u_\eps(t,x)\,dx dt \,=\, \int_{\Omega_T} \eta(t,x)\cdot
  v(t,x)\,d\mu(t,x) \label{eq:conv-v}
\end{gather}
for all $\eta\in C^0_c(\Omega_T,\Rn)$, and such that
\begin{gather}
  \int_{\Omega_T} |v|^2\,d\mu\,\leq\, \liminf_{\eps\to 0}\int_{\Omega_T}
  \eps (\partial_t u_\eps)^2\,dx dt. \label{eq:est-v}
\end{gather}
Moreover,
$(c_0^{-1}\mu^t)_{t\in (0,T)}$ is an $L^2$-flow with generalized velocity $v$ in
the sense of Definition \ref{def:gen-velo}.
\end{lemma}
\begin{proof}
We prove the lemma in three steps.

\noindent \textit{Step 1.} We define approximate velocity vectors $v_\eps:\Omega_T\to\Rn$ by
\begin{gather}
  v_\eps\,:=\,
  \begin{cases}
    -\frac{\partial_t u_\eps}{|\nabla u_\eps|}\frac{\nabla
      u_\eps}{|\nabla u_\eps|} &\text{ if }|\nabla u_\eps|\neq 0,\\[2mm]
    0&\text{ otherwise. }
  \end{cases}
  \label{eq:def-v-eps}
\end{gather} 
and deduce from \eqref{ass:bound-squares} that
\begin{gather}
  \int_{\Omega_T} |v_\eps|^2\,d\tilde{\mu}_\eps \,\leq\, \int_{\Omega_T}
  \eps (\partial_t u_\eps)^2\,dxdt\,\leq\,
  \Lambda + 2\Lambda_0.
  \label{eq:bound-v-eps} 
\end{gather}
Therefore $(\tilde{\mu}_\eps, v_\eps)$ is a measure-function
pair in the sense of \cite{Hutc86}. By \eqref{eq:conv-tilde-mu} and
\cite[Theorem 4.4.2]{Hutc86} we deduce that there 
exists a subsequence $\eps\to 0$ and a function $v\in L^2(\mu,\Rn)$ verifying
\begin{gather*}
-\lim_{\eps\to 0}  \int_{\Omega_T}\eta \cdot \eps\partial_t u_\eps\nabla u_\eps\,dxdt=
\lim_{\eps\to 0}  \int_{\Omega_T} \eta \cdot v_\eps\,d\tilde{\mu}_\eps
=\int_{\Omega_T} \eta\cdot v \,d\mu,
\end{gather*}
for every $\eta\in C^0_c(\Omega_T,\R^n)$, which shows \eqref{eq:est-v}.

\noindent \textit{Step 2.} We claim that $(c_0^{-1}\mu^t)_{t\in(0,T)}$
is an $L^2$-flow with generalized velocity $v$. First we have to verify
\eqref{eq:def_vel_perp}.  With this aim let
\begin{gather*}
  P_\eps\,:=\, Id - \nu_\eps\otimes\nu_\eps.
\end{gather*}
Moreover denote by $P(t,x):\Rn\to T_x\mu^t\subset\R^n$ the orthogonal projection onto
$T_x\mu^t$ whenever this tangential plane exists, and set $P(t,x)$ to be the orthogonal
projection onto $\vec{e}_1^\perp\subset\R^n$ otherwise.
Equation \eqref{eq:def_vel_perp} follows from the identity
\begin{align}
  \int_{\Omega_T} \eta(t,x)\cdot P(t,x)v(t,x)\,d\mu(t,x)
  =\,&\lim_{\eps\to 0}\int_{\Omega_T} \eta(t,x)\cdot
  P_\eps(t,x)v_\eps(t,x)\,d\tilde{\mu}_\eps(t,x)\,=\, 0, \label{eq:conv-moser}
\end{align}
which holds for all $\eta\in C^0_c(\Omega_T)$.
The proof of this statement uses arguments from \cite[Proposition
3.2]{Mose01}. By the varifold convergence of 
$V^t_\eps$ and a Lebesgue-type argument one first obtains strong convergence
(in the sense of \cite{Hutc86}) for the measure function pair
$(\tilde{\mu}_\eps,P_\eps)$. Furthermore,
by \eqref{eq:conv-v} we have weak convergence for the measure function pair
$(\tilde{\mu}_\eps,v_\eps)$ and \eqref{eq:conv-moser} follows, see 
\cite[Lemma 6.3]{MuR08} for details.
\smallskip

\noindent \textit{Step 3.} In order to conclude we have to prove that
\eqref{eq:def_vel_L2_flow} holds.  
By similar calculation as in \eqref{eq:KRT-1} we compute that for any
$\eta\in C^1_c(\Omega_T)$,
\begin{align}
  2\partial_t\mu_\eps^t(\eta(t,\cdot)) \,
  =\, &\int_\Omega \frac{1}{\eps}g_\eps(t,x)^2 \eta(t,x)\, dx - \int_{\Omega}
  \big(\eps(\partial_t u_\eps)^2 
  +\frac{1}{\eps}w_\eps^2\big)(t,x)\eta(t,x) \,dx\notag\\
  &  + 2\int_{\Omega} \partial_t \eta(t,x)\,d\mu_\eps^t(x) -2\int_\Omega
  \eps\nabla\eta(t,x)\cdot \partial_t u_\eps(t,x)\nabla 
  u_\eps(t,x)\, dx. \label{eq:KRT-rep} 
\end{align}
Integrating this equality in time and using \eqref{eq:ass-g},
\eqref{ass:bound-squares} we deduce that
\begin{align}
  &\Big| \int_{\Omega_T} \partial_t \eta(t,x)\,d\mu_\eps(t,x) -
  \int_{\Omega_T}\eps\nabla\eta(t,x)\cdot \partial_t u_\eps(t,x)\nabla
  u_\eps(t,x)\, dx dt \Big| \notag\\
  \leq\,& C(\Lambda,\Lambda_0)\|\eta\|_{C^0(\Omega_T)}.
\end{align}
Due to \eqref{eq:conv-mu} and \eqref{eq:conv-v} this implies
 \eqref{eq:def_vel_L2_flow}. 
\end{proof}
%
\subsection{Convergence to forced mean curvature flow}
\label{sec:proofs-prelim}
We are now
ready to pass to the limit $\eps\to 0$ in \eqref{eq:p-AC}.
\begin{proposition}\label{prop:conv-pAC}
There exists a subsequence $\eps\to 0$ and a function $g\in
L^2(\mu,\Rn)$ such that
\begin{gather}
  \lim_{\eps\to 0}\int_{\Omega_T} -\eta\cdot\nabla u_\eps
  g_\eps\,dx\,dt\,=\, \int_{\Omega_T} \eta\cdot g\,d\mu
  \label{eq:conv-g-rep}
\end{gather}
for all $\eta\in C^0_c(\Omega,\Rn)$ and
\begin{gather}
  \int_{\Omega_T} |g|^2\,d\mu\,\leq\, \liminf_{\eps\to 0}\int_{\Omega_T}
  \frac{1}{\eps} g_\eps^2(x,t)\,dx dt. \label{eq:est-g}
\end{gather}
Moreover $(\mu^t)_{t\in (0,T)}$ and $g$ satisfy
\begin{gather}
  \int_{\Omega_T} \eta\cdot v\,d\mu\,=\, \int_{\Omega_T} \eta\cdot
  (H+g)\,d\mu \label{eq:lim-AC-int}
\end{gather}
for all $\eta\in C^0_c(\Omega,\Rn)$, which implies \eqref{eq:lim-AC}.
Finally the energy inequality
\begin{gather}
  - \int_{\Omega_T} \partial_t \zeta \, d\mu \,\leq\,
  \int_{\Omega_T} v\cdot\nabla\zeta\,d\mu
  -\frac{1}{2}\int_{\Omega_T} (|v|^2 + |H|^2)\zeta \,d\mu \notag\\
  +\frac{1}{2}\|\zeta\|_{C^0(\Omega_T)}\liminf_{\eps\to 0}
  \int_{\Omega_T} \frac{1}{\eps}g_\eps^2\,dx dt. \label{eq:energy-ineq}
\end{gather}
holds for all $\zeta\in C^1_c(\Omega_T)$ with $\zeta\geq 0$.
\end{proposition}
\begin{proof}
From \eqref{eq:ass-g} we deduce that
\begin{gather*}
  \vec{g}_\eps \,:=\, -\frac{g_\eps}{\eps|\nabla u_\eps|}\nu_\eps
\end{gather*}
satsifies for all $\eps>0$
\begin{gather*}
  \int_{\Omega_T} |\vec{g}_\eps|^2\,d\tilde{\mu}_\eps\,\leq\,
  \int_{\Omega_T} \frac{1}{\eps} g_\eps^2\,dx dt\,\leq\, \Lambda.
\end{gather*} 
Since $\tilde{\mu}_\eps\,\to\,\mu$ by \eqref{eq:conv-tilde-mu} we deduce
from \cite{Hutc86} the existence of a subsequence $\eps\to 0$ and of
$g\in L^2(\mu)$ such that \eqref{eq:est-g} and 
\begin{gather*}
  \lim_{\eps\to 0}\int_{\Omega_T}
  \eta\cdot\vec{g}_\eps\,d\tilde{\mu}_\eps\,=\, 
  \int_{\Omega_T} \eta\cdot g\,d\mu \quad\text{ for all }\eta\in C^0_c(\Omega_T)
\end{gather*}
holds. By the definition of
$\tilde{\mu}_\eps$ and $\vec{g}_\eps$ this is equivalent to
\eqref{eq:conv-g-rep}. 
Multiplying \eqref{eq:p-AC} with $-\eta\cdot \nabla u_\eps$ and
integrating yields
\begin{gather*}
  \int_{\Omega_T} -\eta\cdot \eps \partial_t u_\eps \nabla u_\eps \,dx
  dt\,=\, \int_{\Omega_T} \eta\cdot\nabla u_\eps w_\eps - \eta\cdot
  g_\eps\nabla u_\eps\, dx dt.
\end{gather*}
Using \eqref{eq:conv-H}, \eqref{eq:conv-v}, \eqref{eq:conv-g-rep} we 
can pass to the limit $\eps\to 0$ in this equation and arrive at \eqref{eq:lim-AC-int}.

To prove \eqref{eq:energy-ineq} we rewrite \eqref{eq:KRT-rep} and
integrate in time to obtain
\begin{align*}
  \int_{\Omega_T} -\partial_t\zeta\,d\mu_\eps 
  =\, & -\int_\Omega
  \eps\nabla\zeta(t,x)\cdot \partial_t u_\eps(t,x)\nabla 
  u_\eps(t,x)\, dx dt\\ &-\frac{1}{2}
  \int_{\Omega_T}\zeta \Big(\eps(\partial_t u_\eps )^2 +
  \frac{1}{\eps}w_\eps^2\Big)\,dx dt + 
  \int_{\Omega_T}\frac{1}{2\eps}g_\eps^2 \,\zeta\,dx dt.
\end{align*}
By \eqref{eq:conv-mu}, \eqref{eq:conv-v}, and \eqref{eq:L2-H},
\eqref{eq:est-v} we then deduce \eqref{eq:energy-ineq}.
\end{proof}
\section{Applications}
\label{sec:appl}
\subsection{Small perturbations of the Allen--Cahn equation and motion by mean curvature}
For `small perturbations' of the Allen--Cahn equation in the sense that
\begin{gather}
  \int_{\Omega_T} \frac{1}{\eps}g_\eps^2(t,x)\,dx dt\,\to\, 0\quad\text{ as
  }\eps\to 0 \label{eq:small-g}
\end{gather}
Theorem \ref{the:main} implies that we obtain motion by mean curvature
in the limit. This shows a stability of the convergence of
the Allen--Cahn equation to mean curvature flow. In fact, we obtain here
the convergence to an \emph{enhanced motion} in the sense of Ilmanen \cite{Ilma94}.
\begin{proposition}
Assume that \eqref{eq:p-AC}-\eqref{eq:bdry} hold, and that the
perturbations $(g_\eps)_{\eps>0}$ satisfy \eqref{eq:small-g}. Then the
conclusions of Theorem \ref{the:main} hold with $g=0$. Moreover,
$(\mu^t)_{t\in (0,T)}$ is a Brakke motion. If we in addition assume that
the initial data are well-prepared, in the sense that
\begin{gather}
  \lim_{\eps\to 0} \mu_\eps^0 \,\to\, \frac{c_0}{2}|\nabla u(0,\cdot)|,
  \label{eq:prep-id} 
\end{gather}
then the $L^2$-flow $(\frac{1}{c_0}\mu^t)_{t\in (0,T)}$ together with
the current associated to $\partial^*\{u= 1\}$ is an enhanced motion
in the sense of \cite{Ilma94}
with initial condition $\partial^*\{u(0,\cdot)=1\}$.
\end{proposition}
\begin{proof}
By \eqref{eq:small-g} we can apply Theorem \ref{the:main} and by
\eqref{eq:est-g} we obtain that $g=0$. From \eqref{eq:energy-ineq} and
\eqref{eq:lim-AC} we further conclude that
\begin{gather}
  - \int_{\Omega_T} \partial_t \zeta \, d\mu \,\leq\,
  \int_{\Omega_T} H\cdot\nabla\zeta\,d\mu
  -\int_{\Omega_T}\zeta\, |H|^2 \,d\mu \label{eq:Brakke}
\end{gather}
holds for all $\zeta\in C^1_c(\Omega_T)$ with $\zeta\geq 0$. This is a
time-integrated version of Brakkes inequality. Now one derives from
\eqref{eq:Brakke}, following Evans and Spruck \cite[Theorem
7.1]{EvSp95},  that $(\mu_t)_{t\in (0,T)}$ is a Brakke motion.

Further, we have proved in \cite[Proposition 8.2]{MuR08} that there
exists a nonnegative function $p\in L^2(\mu)$ such that
\begin{gather*}
  \frac{c_0}{2}|\nabla' u|\,\leq\, p\mu,
\end{gather*}
where $\nabla'= (\partial_t,\nabla_x)^T$ denotes the time-space gradient
in $\R\times\Rn$. This implies by \eqref{ass:bound-mu} that
\begin{gather}
  \int_{(t,t+\tau)\times \Omega} \frac{c_0}{2}\,d|\nabla' u|(t,x)
  \notag
  \\
  \leq\,
  \|p\|_{L^2(\mu)}\left(\int_{t}^{t+\tau}
  \mu^s(\Omega)\,ds\right)^{1/2}\,\leq\,
  \tau^{\frac{1}{2}}\|p\|_{L^2(\mu)}C(\Lambda,\Lambda_0,T). \label{eq:c-halb}
\end{gather}
By \eqref{eq:spt-mu}, \eqref{eq:prep-id}, \eqref{eq:c-halb}, and since
$(\mu^t)_{t\in (0,T)}$ moves by Brakke motion, we can conclude that
$(\mu^t)_{t\in (0,T)}$ and 
$\partial^*\{u=1\}$ constitute an enhanced motion with initial condition
$\partial^*\{u(0,\cdot)=1\}$. 
\end{proof}
\begin{remark}
For an enhanced motion Ilmanen \cite{Ilma94} proves consistency and
regularity results. In particular, the initial surface 
$\partial^*\{u^0=1\}$ can be perturbed to one whose evolution is smooth
$\Ha^{n}$-almost everywhere in $\Omega_T$.
\end{remark}
\begin{remark}
Actually the conclusions of Proposition \ref{prop:conv-pAC} still hold
under weaker assumptions on the perturbation $g_\eps$, namely it is
sufficient that  $(g_\eps)_{\eps>0}$ satisfies \eqref{eq:ass-g} and 
that \eqref{eq:conv-g} holds
with $g=0$. 
\end{remark}
\subsection{Equation with perturbed double-well potential and a drift term}
In this section we consider \eqref{eq:p-AC} with perturbations
of the form 
\begin{gather}
  g_\eps(t,x)\,=\,  \eps b_\eps(t,x)\cdot\nabla u_\eps(t,x)
  +f_\eps(t,x) \sqrt{2W(u_\eps(t,x))}   \label{eq:app1}
\end{gather}
with $b_\eps:\Omega_T\to \Rn$, $f_\eps:\Omega_T\to\R$ given.
Whereas the first term describes a drift, the term
$f_\eps(t,x) \sqrt{W(u_\eps(t,x))}$ may arise from a perturbation of the
double well potential. Kobayashi \cite{Koba93}, for instance, introduced such a term as
a thermodynamic driving force in a model for dendritic crystal
growth. He proposed a potential of the form
\begin{gather*}
  W_\eps(r,m)\,=\, W(r) + \Big(\frac{2}{3}(r+1)^3
  -2(r+1)^3\Big) m,
\end{gather*}
with $m=\eps f$, where $f$ may be a function depending on other
quantities such as the temperature. This gives 
\begin{gather*}
  \partial_r W_\eps(r,m)\,=\, W'(r) + 2(r^2-1)m
  \,=\, W'(r) + 4 m\sqrt{W(r)}.
\end{gather*}
and yields in \eqref{eq:p-AC} a pertubative term
\begin{gather*}
  g_\eps \,=\, 4 f \sqrt{W(u_\eps)}.
\end{gather*}
Barles and Soner \cite{BaSo98} and Barles, Soner, and Souganidis
\cite{BaSS93} considered phase field
models of Allen--Cahn type  with a perturbation of the form
\eqref{eq:app1}. They proved the 
convergence to forced mean curvature flow in a viscosity solutions
formulations under 
the assumption that $b_\eps=b_\eps(x)$ and $f_\eps=f_\eps(t,x)$ are
uniformly Lipschitz 
continuous in time and space.

Our Theorem \ref{the:main} covers this situation under weaker assumptions on the
regularity of the forcing term.
\begin{proposition}\label{prop:app1}
Consider \eqref{eq:p-AC} with a perturbation of the form \eqref{eq:app1}
and assume that there exists $\Lambda_1>0$ independent of $\eps>0$ such
that 
\begin{gather}
  \int_0^T \sup_{x\in \Omega} \Big(|f_\eps(t,x)|^2 +
  |b_\eps(t,x)|^2\Big)\,dt \,\leq\, \Lambda_1. \label{eq:ass-app1}
\end{gather}
Then the conclusions of Theorem \ref{the:main} hold.
Moreover, we obtain that the limiting forcing term is given by 
\begin{gather}
g(t,x)= -(P(t,x)-Id)b-f(t,x),
\label{eq:lim-drift}
\end{gather}
where $P(t,x):\Rn\to T_x\mu^t$ denotes the orthogonal projection onto the tangential plane 
$T_x\mu^t$ of $\mu^t$ and $b$ and $f$ are determined by
\begin{align}
  \int_{\Omega_T} b\cdot\eta\,d\mu\,&=\, 
  \lim_{\eps\to 0}\int_{\Omega_T} b_\eps\cdot\eta\, \eps|\nabla u_\eps|
  ^2 dx dt, \label{eq:drift}\\
  \int_{\Omega_T} f\cdot\eta\,d\mu\,&=\, 
  \lim_{\eps\to 0}\int_{\Omega_T} f_\eps \sqrt{2W(u_\eps)}\nabla u_\eps\cdot\eta\,dx dt. \label{eq:lim-f}
\end{align}
Finally, in the case that $b_\eps,\, f_\eps$ are continuous 
and converge as $\eps\to 0$  uniformly in $\Omega_T$ to
$\widehat b$   and $\widehat f$, then $b=\widehat b$ on $\spt(\mu)$ and
$f\,\mu= \frac{c_0}{2}\widehat{f}\nabla u$.
\end{proposition}
\begin{proof}
We have to verify that $g_\eps$ satisfies \eqref{eq:ass-g}.
We first show a uniform bound on the diffuse surface
surface area. With this aim we set as above $w_\eps\,=\, -\eps\Delta u_\eps
+\frac{1}{\eps}W'(u_\eps)$ 
and compute
\begin{align*}
  &\frac{d}{dt}\int_\Omega \Big(\frac{\eps}{2}|\nabla u_\eps|^2(t,x) 
  +\frac{1}{\eps}W(u_\eps(t,x))\Big)\,dx
   \\ 
  =\,&\int_\Omega w_\eps(t,x) \partial_t u_\eps(t,x)\,dx 
  \\ 
  =\,& \int_\Omega -\frac{1}{\eps}w_\eps^2(t,x) + w_\eps(t,x)
  \Big(b_\eps(t,x)\cdot\nabla 
  u_\eps(t,x) + \frac{1}{\eps}
  f_\eps(t,x)\sqrt{W(u_\eps(t,x))}\Big)\,dx 
  \\ 
  \leq\,& \int_\Omega \Big(\frac{\eps}{2} |b_\eps(t,x)|^2 |\nabla u_\eps(t,x)|^2 +
  \frac{1}{\eps}f_\eps(t,x)^2 W(u_\eps(t,x))\Big)\,dx
  \\
  \leq\,& \sup_{x\in\Omega} \Big(|f_\eps(t,x)|^2 +
  |b_\eps(t,x)|^2\Big) \int_\Omega \Big(\frac{\eps}{2}|\nabla u_\eps|^2(t,x) 
  +\frac{1}{\eps}W(u_\eps(t,x))\Big)\,dx.
\end{align*}
Hence Gronwall's inequality and \eqref{eq:ass-app1} imply that
\begin{gather}
  E_\eps(u_\eps(t,\cdot)) \,\leq\,
  E_\eps(u_\eps(0,\cdot)) e^{\Lambda_1}. \label{eq:app-E}
\end{gather}
Under the assumption \eqref{eq:ass-ini} on the initial data 
we deduce that
\begin{align}
  \int_{\Omega_T} \frac{1}{\eps} g_\eps(t,x)^2\,dx\,dt \,&\leq\, 2
  \int_{\Omega_T} \Big(\eps|b_\eps(t,x)|^2 |\nabla u_\eps(t,x)|^2 +
  \frac{1}{\eps}f_\eps(t,x)^2 W(u_\eps(t,x))\Big)\,dx\,dt
   \notag
   \\
  \leq\, &4\int_0^T \sup_{x\in \Omega} \Big(|f_\eps(t,x)|^2 +
  |b_\eps(t,x)|^2\Big) E_\eps(u_\eps(t,\cdot))\,dt
  \notag
  \\
  \leq\, &4 \Lambda_0 e^{\Lambda_1}\Lambda_1,
 \label{eq:intermer}
\end{align}
which verifies \eqref{eq:ass-g}. Therefore we obtain \eqref{eq:fmc} with $g$ satisfying 
\eqref{eq:conv-g}. To prove the representation formula \eqref{eq:lim-drift}, with $b,f$ as in
\eqref{eq:drift}, \eqref{eq:lim-f}, we compute
\begin{align}
  \int_{\Omega_T} -\eta\cdot\nabla u_\eps g_\eps\,dx dt \,=\,&
  \int_{\Omega_T} -\eta \cdot \big(P_\eps-Id\big) 
  b_\eps\,d\tilde{\mu}_\eps \notag\\ &-
  \int_{\Omega_T} f_\eps \sqrt{2W(u_\eps)}\nabla u_\eps\cdot\eta\,dx dt \label{eq:fb}
\end{align}
To characterize the limit of the right-hand side of this equation we
first observe that by  \eqref{ass:bound-mu}, \eqref{eq:ass-app1} 
we have
\begin{gather*}
  \sup_{\eps>0}\int_{\Omega_T} |b_\eps|^2 +
  |\vec{f}_\eps|^2\,d\tilde{\mu}_\eps \,\leq\, 2\sup_{\eps>0}
  \int_0^T  \sup_{x\in\Omega}\Big(|b_\eps(t,x)|^2 + f_\eps(t,x)^2\Big)
  E_\eps(u_\eps(t,\cdot))\,dt \,<\, \infty,
\end{gather*}
where we defined
\begin{gather}
  \vec{f}_\eps\,:=\, f_\eps \frac{\sqrt{2 W(u_\eps)}}{\eps|\nabla
      u_\eps|}\nu_\eps \label{eq:def-f-vec}.
\end{gather}
By \eqref{eq:conv-xi} and
\cite[Theorem 4.4.2]{Hutc86} there exist $b,f\,\in\, L^2(\mu)$ and
a subsequence $\eps\to 0$ such that \eqref{eq:drift} and
\begin{gather*}
  \int_{\Omega_T} f\cdot\eta\,d\mu\,=\, \lim_{\eps\to 0} \int_{\Omega_T}
  \vec{f}_\eps\cdot\eta \,d\tilde{\mu}_\eps \,=\, \lim_{\eps\to 0}
  \int_{\Omega_T} \eta\cdot f_\eps \sqrt{2W(u_\eps)}\nabla u_\eps\, dx
  dt,
\end{gather*}
which is \eqref{eq:lim-f}, are satisfied.  
By varifold convergence and an argument similar to that one used in
\cite[Proposition 3.2]{Mose01}, \cite[Lemma 6.3]{MuR08} one obtains that 
$(P_\eps,\tilde{\mu}_\eps)$ converges to $(P,\mu)$ strongly as
measure-function pairs. Together with \eqref{eq:drift} this implies that
\begin{gather*}
  \lim_{\eps\to 0} \int_{\Omega_T} -\eta \cdot \big(P_\eps-Id\big) 
  b_\eps\,d\tilde{\mu}_\eps \,=\, \int_{\Omega_T} -\eta \cdot \big(P-Id\big) 
  b\,d\mu. 
\end{gather*}
By \eqref{eq:lim-f}, \eqref{eq:fb}, this proves the
characterization of $g$.

In the case that $b_\eps\,\to\, \hat{b}$ uniformly in $\Omega_T$ we
obtain from \eqref{eq:drift}  that $b=\hat{b}$ on $\spt(\mu)$. To
characterize $f$ in \eqref{eq:lim-f} we observe that for $G$ as in
\eqref{eq:def-H} 
\begin{gather*}
  \int_{\Omega_T} \eta\cdot f_\eps \sqrt{2W(u_\eps)}\nabla u_\eps\, dx
  dt \,=\, \int_{\Omega_T} \eta\cdot f_\eps \nabla G(u_\eps)\,dx dt.
\end{gather*}
By \cite{Modi87} we have that $\nabla G(u_\eps)\,\to\,
\frac{c_0}{2}\nabla u_\eps$ weakly as measures and since $f_\eps\to f$
uniformly we conclude that
\begin{gather*}
  \int_{\Omega_T} f\cdot\eta\,d\mu\,=\, \frac{c_0}{2}\int_{\Omega_T}
  f\eta\cdot \nabla u.
\end{gather*}
\end{proof}
\begin{remark}
By the same arguments we also can allow for pertubations of the form
$g_\eps=\eps f_\eps |\nabla u_\eps|$ as were considered by Benes and
Mikula \cite{BeMi98} in a
model for phase transitions and by Benes, Chalupeck\'{y}, and Mikula in
image processing \cite{BeCM04}.
\end{remark}

\subsection{Application to Mullins-Sekerka problem with kinetic undercooling}  
Here we apply our main Theorem in a situation
where the forcing term in the limit is not concentrated on the phase
interface but rather given by the trace of a Sobolev
function in the ambient space. As a concrete application we prove the
convergence of phase field approximations of the Mullins--Sekerka
problem with kinetic undercooling. This improves in space dimensions
$n=2,3$ an earlier result by Soner \cite{Sone95}.
Throughout this section we will assume $\Omega=\R^n$. 
As noticed in Remark \ref{rem:unbounded} our main results apply also to
this case. Let us consider the Allen--Cahn equation with perturbations 
$g_\eps$ that are given by
\begin{gather}
  g_\eps(t,x)\,=\,\theta_\eps(t,x)\sqrt{2W(u_\eps(t,x))},
  \label{eq:surf-theta}
\end{gather}
where we now assume that $\theta_\eps(t,\cdot)\in C^1(\Omega)$ for all $t\in \R$ and that
\begin{gather}
  \sup_{\eps>0}\int_{\Omega_T}\big(\theta_\eps^2+\vert\nabla\theta_\eps\vert^2\big)
  \,dxdt\,<\,\infty. 
  \label{eq:bulk-Sobolev}
\end{gather}
We first show that we can derive from this control of $\theta_\eps$ in
the bulk that the assumption \eqref{eq:ass-g}, which was necessary to apply Theorem
\ref{the:main}, is satisfied by $g_\eps$.
\begin{proposition}\label{prop:g-bulk} 
Let sequences $(u_\eps)_{\eps>0}$, $(\theta_\eps)_{\eps>0}$ be given and define $g_\eps$ by
\eqref{eq:surf-theta}. Assume that \eqref{eq:bulk-Sobolev} is satisfied and
that $u_\eps,g_\eps$ are solutions of \eqref{eq:p-AC}-\eqref{eq:bdry}. Furthermore let
\eqref{eq:ass-ini} hold for the initial data $u_\eps^0$ and assume that we
have a uniform upper bound on the density of the diffuse surface area measures,
\begin{gather}
  \sup_{x\in\R^n,\,R>0}\frac{\mu_\eps^t(B_R(x))}{R^{n-1}}\leq K(T), \quad \forall\, t\in\,[0,T).
  \label{eq:upper-bound-mu-eps-t}
\end{gather}
Then $g_\eps$ satisfies \eqref{eq:ass-g}.
\end{proposition}
\begin{proof}
By \cite[Theorem 5.12.4]{Ziem82} it follows from
\eqref{eq:upper-bound-mu-eps-t} that for any $\varphi\in C^1_c(\Rn)$
\begin{gather}
  \left\vert\int_{\R^n}\varphi(x)\,d\mu(x)\right\vert \,\leq\, K_n\,
  M(\mu)\int_{\R^n}\vert\nabla\varphi\vert\,dx. \label{eq:est-Ziemer}
\end{gather}
For $R>1$ we choose a smooth cut-off function $\varphi_R\in C^1_c(\R^n)$
with $\varphi_R\geq 0$ on $\R^n$, $\varphi_R\equiv 1$ on $B_R$,
$\varphi_R\equiv 0$ on $B_{2R}$, and
$\|\nabla\varphi_R\|_{L^\infty(\R^n)}\leq 1$. We then obtain that
\begin{gather}
  \frac{1}{\eps}\int_{B_R}g_\eps^2(t,x)\,dx \,=\,
  \int_{B_R}\theta_\eps^2(t,x)\frac{W(u_\eps(t,x))}{\eps}\,dx 
  \leq
  \int_{\R^n}\varphi_R(x)\theta_\eps^2(t,x)\,d\mu_\eps^t(x).
  \label{eq:proof-Z-1} 
\end{gather}
Applying \eqref{eq:est-Ziemer} with
$\varphi=\varphi_R\theta_\eps^2(t,\cdot)$ we deduce that the
right-hand-side of \eqref{eq:proof-Z-1} is estimated by
\begin{align*}
  \int_{\R^n}\varphi_R(x)\theta_\eps^2(t,x)\,d\mu_\eps^t(x)
  \,&\leq\, K_n\,
  K(T)\int_{\R^n}\vert\nabla\big(\varphi_R\theta_\eps^2(t,\cdot)\big)\vert\,dx 
  \\
  &\leq\,  K_n\,
  K(T)\int_{\R^n}\theta_\eps^2(t,\cdot)\vert\nabla\varphi_R\vert +
  2\varphi_R\vert\theta_\eps(t,\cdot)\nabla\theta_\eps(t,\cdot)\vert\,dx 
  \\
  &\leq\,  K_n\,
  K(T)\int_{\R^n}\big(2\theta_\eps^2(t,\cdot)+\vert\nabla\theta_\eps(t,\cdot)\vert^2\big)\,dx.   
\end{align*}
With $R\to\infty$ we deduce from \eqref{eq:proof-Z-1}, the last
inequality, and \eqref{eq:bulk-Sobolev} that \eqref{eq:ass-g} holds. 
\end{proof}
Next we apply Proposition \ref{prop:g-bulk} to the phase-fields
approximation of the Mullins-Sekerka problem with kinetic undercooling
introduced in \cite{Sone95}. More precisely, let
$(u_\eps,\theta_\eps)_{\eps>0}$ be the unique, bounded, smooth solutions
on $Q:=(0,+\infty)\times\R^n$ to the following Cauchy problem
\begin{align}
  \eps\partial_t u_\eps \,&=\, \eps\Delta
  u_\eps-\frac{W^\prime(u_\eps)}{\eps}+\sqrt{2 W(u_\eps)}\theta_\eps
  &&\text{ in }Q,
  \label{eq:front}
  \\
  \partial_t \theta_\eps \,&=\, \Delta\theta_\eps- \sqrt{2 W(u_\eps)}\partial_t u_\eps
  &&\text{ in }Q, \label{eq:bulk}\\
  u_\eps(0,\cdot) \,&=\, u^0_\eps, \quad
  \theta_\eps(0,\cdot)\,=\,\theta^0_\eps, &&\text{ in }\R^n. 
\label{eq:initial-cond}
\end{align}
\begin{proposition}\label{prop:mullins-sek}
Let $n=2,3$ and let $(u_\eps,\theta_\eps)_{\eps>0}$ satisfy
\eqref{eq:front}-\eqref{eq:initial-cond}. Assume that the initial data
$u^0_\eps,\theta^0_\eps$ 
are well-prepared in the sense of \cite[Section 2.4]{Sone95} and that
\begin{gather}
  \sup_{\eps>0}\| u^0_\eps\|_{L^\infty(\R^n)}\leq 1,
  \label{eq:ass-initial-cond-I}
  \\
  \sup_{\eps>0} \int_{\R^n}\left(\frac{\eps}{2}\vert \nabla u^0_\eps\vert^2
    +\frac{W(u^0_\eps)}{\eps}+(\theta^0_\eps)^2\right)\,dx<C_1
  \label{eq:ass-initial-cond-II}
\end{gather}
hold. Then there exists a subsequence $\eps\to 0$ (not relabelled) and
functions
\begin{align*}
  \theta \,&\in\, L^\infty_{loc}(0,\infty;L^2(\Rn))\cap
  L^2_{loc}(0,\infty;H^{1,2}(\Rn)),\\ 
  u \,&\in\, BV_{loc}(Q)\cap L^\infty(0,\infty;BV_{loc}(\R^n,\{-1,1\}),
\end{align*}
such that
\begin{align}
   \theta_\eps\,&\to\, \theta &&\text{ weakly in
   }L^2_{loc}(0,\infty;H^{1,2}(\Rn)), \label{eq:conv-theta-1}\\
   \theta_\eps(t,\cdot)\,&\to\, \theta(t,\cdot) &&\text{in }L_{loc}^2(\R^n),
   \text { for every }t\geq 0, \label{eq:conv-theta-2}
\end{align}
and
\begin{gather}
  u_\eps\,\to\, u \qquad\text{ in }L^p_\loc(Q) \text{ for all }1\leq
  p<\infty. \label{eq:conv-u-MS} 
\end{gather}
Moreover the conclusion of Theorem \ref{the:main} hold. In particular,
there exists a measurable function 
$\alpha:\partial^*\{u=1\}\to\N$ such that, in the generalized formulation
of Theorem \ref{the:main},
\begin{gather}
  v\,=\, H - \frac{1}{\alpha}\theta\nu \qquad \Ha^n-\text{almost
  everywhere on }\partial^*\{u=1\},
  \label{eq:MS-fmc}
\end{gather}
where $\nu$ denotes the inner normal of $\{u=1\}$ on $\partial^*\{u=1\}$.
%

Finally we have for every $\eta\in C^\infty_c(Q)$ that
\begin{equation}
  -\int_Q(\partial_t\eta+\Delta\eta)\theta\,dxdt\,=\, \frac{c_0}{2}\int_Q
   u\partial_t\eta \,dxdt.  \label{eq:MS-bulk} 
\end{equation}
\end{proposition}
\begin{remark} An example of well-prepared intial data is given by
$\theta^0_\eps=\theta^0\in C^2_c(\R^n)$, and $u^0_\eps(x):= 
\gamma_\eps(d_{\partial E}(x)/\eps)$, where $d_{\partial E}$ denotes the
signed distance from the smooth boundary of the open subset $E$ of
$\R^n$, while $\gamma_\eps$ are appropriate approximations of the
optimal transition profile $\tanh (\cdot/\sqrt{2})$. 
\end{remark}
\begin{proof}[Proof of Proposition \ref{prop:mullins-sek}]
We show that Proposition \ref{prop:g-bulk} can be
applied. The key estimates have been proved in
\cite{Sone95}. Firstly by \eqref{eq:ass-initial-cond-I} and the maximum 
principle $\vert u_\eps(t,x)\vert\leq 1$ for every $(t,x)\in Q$. Hence
\eqref{eq:surf-theta} is satifsfied.
Next, solutions of  \eqref{eq:front}-\eqref{eq:initial-cond} satisfy the 
energy identity 
%
\begin{gather}
  \mu_\eps^t(\R^d) + \|\theta_\eps(t,\cdot)\|_{L^2(\R^n)}^2 +
  \int_0^t\int_{\R^n}\big(\eps(\partial_t u_\eps)^2+\vert\nabla\theta_\eps\vert^2\big)\,dxdt
  \,=\, \mu_\eps^0(\R^d)+\|\theta_\eps^0\|_{L^2(\R^n)} \label{eq:MS-energy}
\end{gather}
and \eqref{eq:conv-theta-1}-\eqref{eq:conv-u-MS} follow from
\eqref{eq:MS-energy} and \cite[Section 2.3, Proposition
3.4]{Sone95}. Finally the crucial estimate 
\eqref{eq:upper-bound-mu-eps-t} was shown in \cite[Proposition 7.2]{Sone95}. 
By Remark \ref{rem:unbounded} and Proposition \ref{prop:g-bulk} we now
can apply Theorem 
\ref{the:main}. To derive the limiting forcing term $g$ we consider the
function $G$ defined in \eqref{eq:def-H}
and observe that that by \eqref{eq:conv-g} for all $\eta\in C^1_c(Q)$
\begin{align}
  \int_Q\eta\,\cdot\,g\,d\mu \,=\,&\lim_{\eps\to 0}
  \int_Q -\eta\,\cdot\,\theta_\eps \sqrt{2W(u_\eps)}\nabla u_\eps\,dxdt
  \notag\\ 
  =\,& \lim_{\eps\to 0}\int_Q -\eta\cdot\theta_\eps\nabla G(u_\eps)\,dxdt
  \notag\\
  =\, &\lim_{\eps\to 0} \int_Q G(u_\eps)\nabla\cdot (\eta \theta_\eps)\,dxdt
  \,=\,  \frac{c_0}{2} \int_Q u \nabla\cdot (\eta \theta)\,dxdt, \label{eq:MS-fmc-weak}
\end{align}
where we have used \eqref{eq:conv-theta-1}, \eqref{eq:conv-u-MS}. 
Now let $\alpha(t,x):= \theta^{n-1}(c_0^{-1}\mu^t,x)$ denote the
$(n-1)$-dimensional density of $c_0^{-1}\mu^t$ in $x$. By the integrality of
$c_0^{-1}\mu^t$ we obtain that $\alpha$ is integer valued $\Ha^n$-almost
everywhere. From \eqref{eq:MS-fmc-weak} we deduce that
\begin{gather*}
  gc_0 \alpha \,=\, -c_0 \theta \nu
\end{gather*}
$\Ha^{n}$-almost everywhere on $\partial^*\{u=1\}$ and \eqref{eq:MS-fmc}
follows from \eqref{eq:lim-AC}.

To derive \eqref{eq:MS-bulk} we first multiply \eqref{eq:bulk} with
$\eta\in C^1_c(Q)$, integrate over $Q$, do some partial integrations, and
use that $\sqrt{2W(u_\eps)}\partial_t u_\eps = \partial_t
G(u_\eps)$. This gives
\begin{gather*}
  \int_Q \partial_t\eta \theta_\eps\,dx dt \,=\, \int_Q
  \nabla\eta\cdot\nabla \theta_\eps - \partial_t \eta  G(u_\eps)\,dxdt
\end{gather*}
and by \eqref{eq:conv-theta-1}, \eqref{eq:conv-u-MS} we conclude that
\eqref{eq:MS-bulk} holds.
\end{proof}
\begin{remark}
Proposition \ref{prop:mullins-sek} improves the results obtained in
\cite{Sone95} for space dimensions $n\leq 3$. Firstly we have
shown that $c_0^{-1}\mu^t$ are for $\LL^1$-a.e. $t\in (0,+\infty)$
\textit{integer  rectifiable}, which implies  by \cite[Section
5.8]{Brak78} that the generalized mean curvature vector $H(t,\cdot)$ is
$\mu^t$-a.e. orthogonal to $T_x\mu^t$.
Secondly we are closer to a pointwise formulation of the interface
motion law on the phase boundary. 
The occurrence of an integer factor $\alpha$ in \eqref{eq:MS-fmc} is
typical in the varifold approach to the convergence in phase field
equations, see for example \cite{Chen96} and \cite{Tone05}, or
\cite{RTo08} for a situation where this problem could be resolved.
\end{remark}

\subsection{Application to a model for diffusion induced grain boundary
  motion}
As another application we discuss a model for diffusion
induced grain boundary motion proposed by Cahn, Fife, and Penrose
\cite{CaFP97} that was analyzed in a couple of different papers
\cite{FiCE01,DeEl01,DeES01,GaNS07}. The model describes the dynamics of
two phases of different orientations in a polycrystalline film and of the
concentration of certain atoms that diffuse along the grain
boundaries from outside into the film. The system is driven by the
reduction of surface area of the grain boundary and a driving force
that depends on the concentration of the metal.
In the model a free surface area of the form
\begin{gather}
  \F_\eps(u_\eps,c_\eps)\,=\, \int_\Omega \frac{\eps}{2}|\nabla
  u_\eps|^2 +\frac{1}{\eps}W(u_\eps) + \frac{1}{2\eps} c_\eps^2 + (u_\eps+1)
  f(c_\eps)\,dx \label{eq:F-app2}
\end{gather}
is considered, where $u_\eps$ is a phase field that describes two
different crystal-lattice orientations indicated by the values $u_\eps=\pm 1$ 
and where $c_\eps$ denotes the concentration field of
atoms. Here $W$ is a double well potential (typically a double obstacle
potential) and $f$ is a given globally Lipschitz continuous function with $f(0)=0$ and
$f(r)\geq 0$ for $r\in [0,1]$, which is the range of physically
meaningful values for $c_\eps$. 
A gradient flow dynamic is assumed, with respect to  a scalar product
that is of $L^2$-type in the first component and
of $H^{-1}$-type with degenerate
mobility in the second component,
\begin{gather*}
  \|(v_1,v_2)\|^2_{(u_\eps,c_\eps)}\,:=\, \|v_1\|_{L^2(\Omega)}^2 + \int_\Omega
  D(u_\eps)|\nabla Z(v_2)|^2
\end{gather*}
for a tangent vector $(v_1,v_2)$ at $(u_\eps,c_\eps)$ and for a suitable
degenerate mobility function $D$. Here $z=Z(v_2)$ denotes the solution of
\begin{gather*}
  -\nabla \big(D(u_\eps)\nabla z\big)\,=\, v_2.
\end{gather*}
These choices lead to a system of equations 
\begin{align}
  \eps\partial_t u_\eps\,&=\, \eps\Delta u_\eps
  -\frac{1}{\eps}W'(u_\eps) + f(c_\eps), \label{eq:1app2}\\
  \eps \partial_t c_\eps \,&=\,
  \nabla\cdot\Big(D(u_\eps)\nabla\big(c_\eps +\eps\, (u_\eps+1)
  f'(c_\eps)\big)\Big), \label{eq:2app2}
\end{align}
complemented by
initial and boundary conditions, 
\begin{align}
  u_\eps(0,\cdot)\,&=\,u_\eps^{0},\quad c_\eps(0,\cdot)\,=\,
  c_\eps^{0}\qquad\text{ in }\Omega \label{eq:app2-ini}\\
  \nabla u_\eps\cdot \nu_\Omega\,&=\,0,\quad D(u_\eps)\nabla
  c_\eps\cdot\nu_\Omega\,=\, 0 \qquad\text{ on
  }(0,T)\times\partial\Omega. \label{eq:app2-bdry}
\end{align}
We choose in the following $W$ to be the standard quartic double well potential,
assume that $f$ is linear and consider the convergence as $\eps\to
0$ of \eqref{eq:1app2} only.
\begin{proposition}\label{prop:app2}
Let a sequence $(u_\eps,c_\eps)_{\eps>0}$ of solutions
of \eqref{eq:1app2}-\eqref{eq:app2-bdry} be given, let $f(r)=r$, and
assume that the initial data satisfy
\begin{gather}
  \F_\eps(u_\eps^{0},c_\eps^0) \,\leq\, \Lambda_0 \label{eq:ini-app2}
\end{gather}
for all $\eps>0$. Then there exists a
subsequence $\eps\to 0$, 
a phase indicator function $u\in BV(\Omega_T)\cap
L^\infty(0,T;BV(\Omega;\{-1,1\})$, a $L^2$-flow $(\mu^t)_{t\in
(0,T)}$, and a function $c\in L^2(\mu;\Rn)$ such that the conclusions of
Theorem \ref{the:main} hold. In particular we obtain that for all
$\eta\in C^0_c(\Omega_T)$  
\begin{gather}
  \lim_{\eps\to 0}\int_{\Omega_T} \eta\cdot\nabla u_\eps
  c_\eps\,dx\,dt\,=\, \int_{\Omega_T} \eta\cdot c\,d\mu
  \label{eq:conv-app2-c}
\end{gather}
and that
\begin{gather}
  H \,=\, v + c\label{eq:lim-AC-app2}
\end{gather}
holds $\mu$-almost everywhere.
\end{proposition}
\begin{proof}
Let us firstly notice that for every  for every $\eps$ small enough
\begin{gather*}
  \frac{1}{2\eps}W(u_\eps)+\frac{1}{4\eps}c_\eps^2\,
  \leq\, 1 + \frac{1}{\eps}W(u_\eps)+ \frac{1}{2\eps}c_\eps^2 +\, c_\eps
  (1+u_\eps)
\end{gather*}
Since $t\mapsto
\F_\eps(u_\eps(t,\cdot),c_\eps(t,\cdot))$ is nonincreasing
under the gradient flow dynamics we obtain from the above inequality
and \eqref{eq:ini-app2} that
\begin{gather*}
  \frac{1}{2} E_\eps(u_\eps(t,\cdot)) + \frac{1}{4}\int_\Omega
  \frac{1}{\eps}c_\eps^2(t,x)\,dx \, \leq \LL^n(\Omega) +
  \mathcal\F_\eps(u_\eps(t,\cdot),c_\eps(t,\cdot)) \,\leq\,
  \Lambda_0+\LL^n(\Omega). 
\end{gather*}
This latter inequality furnishes the uniform bound needed
to apply Theorem \ref{the:main}, and the conclusions follow.
\end{proof}
\begin{remark}
We do not address here the questions of convergence in
\eqref{eq:2app2} and of the right choice for the mobility function $D$
in the case of a quartic double-well potential.
We expect that taking \eqref{eq:2app2} into account will improve the
convergence of $c_\eps$ and in \eqref{eq:1app2} and will allow to prove
\eqref{eq:2app2} for more general Lipschitz functions $f$.
\end{remark}
\def\cprime{$'$}


\begin{thebibliography}{10}

\bibitem{Alla72}
W.~K. Allard.
\newblock {On the first variation of a varifold}.
\newblock {\em Annals of Mathematics}, 95:417--491, 1972.

\bibitem{BaSS93}
G.~Barles, H.~M. Soner, and P.~E. Souganidis.
\newblock Front propagation and phase field theory.
\newblock {\em SIAM J. Control Optim.}, 31(2):439--469, 1993.

\bibitem{BaSo98}
G.~Barles and P.~E. Souganidis.
\newblock A new approach to front propagation problems: theory and
  applications.
\newblock {\em Arch. Rational Mech. Anal.}, 141(3):237--296, 1998.

\bibitem{BeNo97}
G.~Bellettini and M.~Novaga.
\newblock Minimal barriers for geometric evolutions.
\newblock {\em J. Differential Equations}, 139(1):76--103, 1997.

\bibitem{BePa95}
G.~Bellettini and M.~Paolini.
\newblock Some results on minimal barriers in the sense of {D}e {G}iorgi
  applied to driven motion by mean curvature.
\newblock {\em Rend. Accad. Naz. Sci. XL Mem. Mat. Appl. (5)}, 19:43--67, 1995.

\bibitem{BeCM04}
M.~Benes, V.~Chalupeck{\'y}, and K.~Mikula.
\newblock Geometrical image segmentation by the allen-cahn equation.
\newblock {\em Applied Numerical Mathematics}, 51(2-3):187--205, Nov 2004.

\bibitem{BeMi98}
M.~Bene{\v{s}} and K.~Mikula.
\newblock Simulation of anisotropic motion by mean curvature---comparison of
  phase field and sharp interface approaches.
\newblock In {\em Proceedings of the {A}lgoritmy'97 {C}onference on
  {S}cientific {C}omputing ({Z}uberec)}, volume~67, pages 17--42, 1998.

\bibitem{Brak78}
K.~A. Brakke.
\newblock {\em The motion of a surface by its mean curvature}, volume~20 of
  {\em Mathematical Notes}.
\newblock Princeton University Press, Princeton, N.J., 1978.

\bibitem{CaFP97}
J.~Cahn, P.~Fife, and O.~Penrose.
\newblock A phase-field model for diffusion-induced grain-boundary motion.
\newblock {\em Acta Materialia}, 45(10):4397--4413, 1997.
\newblock DOI: 10.1016/S1359-6454(97)00074-8.

\bibitem{ChNo08}
A.~Chambolle and M.~Novaga.
\newblock Implicit time discretization of the mean curvature flow with a
  discontinuous forcing term.
\newblock {\em Interfaces and Free Boundaries}, pages 283--300, 2008.

\bibitem{Chen96}
X.~Chen.
\newblock Global asymptotic limit of solutions of the {C}ahn-{H}illiard
  equation.
\newblock {\em J. Diff. Geom.}, 44:262--311, 1996.

\bibitem{ChGG91}
Y.~G. Chen, Y.~Giga, and S.~Goto.
\newblock {\em Uniqueness and existence of viscosity solutions of generalized
  mean curvature flow equations}.
\newblock J. Diff. Geom. 33, 749-786, 1991.

\bibitem{CrIL92}
M.~G. Crandall, H.~Ishii, and P.-L. Lions.
\newblock User's guide to viscosity solutions of second order partial
  differential equations.
\newblock {\em Bull. Amer. Math. Soc. (N.S.)}, 27(1):1--67, 1992.

\bibitem{DaMo81}
G.~Dal~Maso and L.~Modica.
\newblock A general theory of variational functionals.
\newblock In {\em Topics in functional analysis, 1980--81}, Quaderni, pages
  149--221. Scuola Norm. Sup. Pisa, Pisa, 1981.

\bibitem{DeG91}
E.~De~Giorgi.
\newblock Some remarks on {$\Gamma$}-convergence and least squares method.
\newblock In {\em Composite media and homogenization theory (Trieste, 1990)},
  volume~5 of {\em Progr. Nonlinear Differential Equations Appl.}, pages
  135--142. Birkh\"auser Boston, Boston, MA, 1991.

\bibitem{DeEl01}
K.~Deckelnick and C.~M. Elliott.
\newblock An existence and uniqueness result for a phase-field model of
  diffusion-induced grain-boundary motion.
\newblock {\em Proc. Roy. Soc. Edinburgh Sect. A}, 131(6):1323--1344, 2001.

\bibitem{DeES01}
K.~Deckelnick, C.~M. Elliott, and V.~Styles.
\newblock Numerical diffusion-induced grain boundary motion.
\newblock {\em Interfaces Free Bound.}, 3(4):393--414, 2001.

\bibitem{DiLN01}
N.~Dirr, S.~Luckhaus, and M.~Novaga.
\newblock A stochastic selection principle in case of fattening for curvature
  flow.
\newblock {\em Calc. Var. Partial Differential Equations}, 13(4):405--425,
  2001.

\bibitem{EvSS92}
L.~C. Evans, H.~M. Soner, and P.~E. Souganidis.
\newblock Phase transitions and generalized motion by mean curvature.
\newblock {\em Comm. Pure Appl. Math.}, 45(9):1097--1123, 1992.

\bibitem{EvSp91}
L.~C. Evans and J.~Spruck.
\newblock {\em Motion of level sets by mean curvature. I}.
\newblock J. Diff. Geom. 33, 635-681, 1991.

\bibitem{EvSp95}
L.~C. Evans and J.~Spruck.
\newblock Motion of level sets by mean curvature. {IV}.
\newblock {\em J. Geom. Anal.}, 5(1):77--114, 1995.

\bibitem{FiCE01}
P.~C. Fife, J.~W. Cahn, and C.~M. Elliott.
\newblock A free-boundary model for diffusion-induced grain boundary motion.
\newblock {\em Interfaces Free Bound.}, 3(3):291--336, 2001.

\bibitem{GaNS07}
H.~Garcke, R.~N{\"u}rnberg, and V.~Styles.
\newblock Stress- and diffusion-induced interface motion: modelling and
  numerical simulations.
\newblock {\em European J. Appl. Math.}, 18(6):631--657, 2007.

\bibitem{Huis90}
G.~Huisken.
\newblock Asymptotic behavior for singularities of the mean curvature flow.
\newblock {\em J. Differential Geom.}, 31(1):285--299, 1990.

\bibitem{Hutc86}
J.~E. Hutchinson.
\newblock Second fundamental form for varifolds and the existence of surfaces
  minimising curvature.
\newblock {\em Indiana Univ. Math. J.}, 35:45--71, 1986.

\bibitem{HuTo00}
J.~E. Hutchinson and Y.~Tonegawa.
\newblock Convergence of phase interfaces in the van der
  {W}aals-{C}ahn-{H}illiard theory.
\newblock {\em Calc. Var. Partial Differential Equations}, 10(1):49--84, 2000.

\bibitem{Ilma93}
T.~Ilmanen.
\newblock Convergence of the {A}llen-{C}ahn equation to {B}rakke's motion by
  mean curvature.
\newblock {\em J. Differential Geom.}, 38(2):417--461, 1993.

\bibitem{Ilma94}
T.~Ilmanen.
\newblock Elliptic regularization and partial regularity for motion by mean
  curvature.
\newblock {\em Mem. Amer. Math. Soc.}, 108, 1994.

\bibitem{Koba93}
R.~Kobayashi.
\newblock Modeling and numerical simulations of dendritic crystal growth.
\newblock {\em Physica D: Nonlinear Phenomena}, 63(3-4):410--423, 1993.

\bibitem{KORV07}
R.~Kohn, F.~Otto, M.~G. Reznikoff, and E.~Vanden-Eijnden.
\newblock Action minimization and sharp-interface limits for the stochastic
  {A}llen-{C}ahn equation.
\newblock {\em Comm. Pure Appl. Math.}, 60(3):393--438, 2007.

\bibitem{Menn08}
U.~Menne.
\newblock Second order rectifiability of integral varifolds of locally bounded
  first variation, 2008.

\bibitem{Modi87}
L.~Modica.
\newblock The gradient theory of phase transitions and the minimal interface
  criterion.
\newblock {\em Arch. Rational Mech. Anal.}, 98:357--383, 1987.

\bibitem{MoMo77}
L.~Modica and S.~Mortola.
\newblock Un esempio di {$\Gamma$}-convergenza.
\newblock {\em Boll. Un. Mat. Ital. B (5)}, 14(1):285--299, 1977.

\bibitem{Mose01}
R.~Moser.
\newblock A generalization of {R}ellich's theorem and regularity of varifolds
  minimizing curvature.
\newblock Preprint~72, Max Planck Institute for Mathematics in the Sciences,
  Leipzig, 2001.

\bibitem{MuR08}
L.~Mugnai and M.~R{\"o}ger.
\newblock The {A}llen-{C}ahn action functional in higher dimensions.
\newblock {\em Interfaces Free Bound.}, 10(1):45--78, 2008.

\bibitem{RSc06}
M.~R\"oger and R.~Sch\"atzle.
\newblock On a modified conjecture of {D}e {G}iorgi.
\newblock {\em Mathematische Zeitschrift}, 254(4):675--714, 2006.

\bibitem{RTo08}
M.~R{\"o}ger and Y.~Tonegawa.
\newblock Convergence of phase-field approximations to the {G}ibbs-{T}homson
  law.
\newblock {\em Calc. Var. Partial Differential Equations}, 32(1):111--136,
  2008.

\bibitem{Simo83}
L.~Simon.
\newblock {\em Lectures on geometric measure theory}, volume~3 of {\em
  Proceedings of the Centre for Mathematical Analysis, Australian National
  University}.
\newblock Australian National University Centre for Mathematical Analysis,
  Canberra, 1983.

\bibitem{Sone95}
H.~M. Soner.
\newblock Convergence of the phase-field equations to the {M}ullins-{S}ekerka
  problem with kinetic undercooling.
\newblock {\em Arch. Rational Mech. Anal.}, 131(2):139--197, 1995.

\bibitem{Tone05}
Y.~Tonegawa.
\newblock A diffused interface whose chemical potential lies in a {S}obolev
  space.
\newblock {\em Ann. Sc. Norm. Super. Pisa Cl. Sci. (5)}, 4(3):487--510, 2005.

\bibitem{Wa93}
S.-L. Wang and et~al.
\newblock Thermodynamically-consistent phase-field models for solidification.
\newblock {\em Phys. D}, 69(1-2):189--200, 1993.

\bibitem{Ziem82}
W.~Ziemer.
\newblock Interior and boundary continuity of weak solutions of degenerate
  parabolic equations.
\newblock {\em Trans. Amer. Math. Soc.}, 271:733--748, 1982.

\end{thebibliography}
\end{document}